\documentclass[12pt]{amsart}
\usepackage{amsfonts}
\usepackage{graphicx}
\usepackage{color}
\theoremstyle{plain}
\newtheorem{thm}{Theorem}[section]

\newtheorem{cor}[thm]{Corollary}

\newtheorem{que}[thm]{Question}

\theoremstyle{definition}

\newtheorem{const}[thm]{Construction}

\newtheorem{remark}[thm]{Remark}
\begin{document}

\title[Homogeneity Groups of Ends of $3$-manifolds]
{Homogeneity Groups of Ends of Open $3$-Manifolds}
\author{Dennis J. Garity}
\address{Mathematics Department, Oregon State University,
Corvallis, OR 97331, U.S.A.}
\email{garity@math.oregonstate.edu}
\urladdr{http://www.math.oregonstate.edu/$\sim$garity}

\author{Du\v{s}an Repov\v{s}}
\address{Faculty of Education,
and Faculty of Mathematics and Physics,
University of Ljubljana,  P.O.Box 2964,
Ljubljana, Slovenia}
\email{dusan.repovs@guest.arnes.si}
\urladdr{http://www.fmf.uni-lj.si/$\sim$repovs}

\date{\today}



\begin{abstract} 
For every finitely generated abelian group $G$, we construct an
irreducible open $3$-manifold $M_{G}$ whose end  set is
homeomorphic to a Cantor set  and with end homogeneity group of
$M_{G}$ isomorphic to  $G$. The end homogeneity group is the
group of self-homeomorphisms of the end set that extend to
homeomorphisms of the $3$-manifold. The techniques involve
computing the embedding homogeneity groups of carefully
constructed  Antoine type Cantor sets made up of rigid pieces. In
addition, a generalization of an Antoine Cantor set using
infinite chains is needed to construct an example with  integer
homogeneity group. Results about local genus of points in Cantor
sets and about geometric index are also used.
\end{abstract}
\maketitle


\section{Introduction}
\label{introsection}

Each Cantor set $C$ in $S^{3}$ has complement an open
$3$-manifold $M^{3}$  with end set $C$. Properties of the
embedding of the Cantor set give rise to properties of the
corresponding complementary $3$-manifold $M^{3}$. See
\cite{SoSt12}, \cite{GaRe12}, and \cite{GaReWr12} for examples of
this.

We investigate possible group actions on the end set $C$ of the
open $3$-manifold $M^{3}$ in the following sense:
The
\emph{homogeneity group of the end set }is the group of
homeomorphisms of the end set $C$ that extend to homeomorphisms
of the open $3$-manifold $M^{3}$. Referring specifically to the
embedding of the Cantor set, this group can also be called the
\emph{embedding homogeneity group of the Cantor set}. See 
\cite{Di11} and \cite{vM11} for a discussion and overview of some
other types of homogeneity.

The standardly embedded Cantor set is at one extreme here. The
embedding homogeneity group is the full group of 
self-homeomorphisms of the Cantor set, an extremely rich group (there
is such a homeomorphism taking any countable dense set to any
other). Cantor sets with this full embedding homogeneity group
are called \emph{strongly homogeneously embedded.} See Daverman
\cite{Da79} for an example of a non-standard Cantor set with this
property.

At the other extreme are \emph{rigidly embedded} Cantor sets,
i.e. those Cantor sets for which only the identity homeomorphism
extends. Shilepsky \cite{Shi74} constructed Antoine type
\cite{An20} rigid Cantor sets. Their rigidity is a consequence of
Sher's result \cite{Sh68} that if two Antoine Cantor sets are
equivalently embedded, then the stages of their defining
sequences must match up exactly. In the last decade, new examples
\cite{GaReZe06, GaReWrZe11} of non-standard Cantor sets were
constructed  that were both rigidly embedded and had simply
connected complement. See \cite{Wr86} for  additional examples of rigidity.

These examples naturally lead to the question of which types of
groups can arise as end homogeneity groups between the two
extremes mentioned above. In this paper we show that for each
finitely generated abelian group $G$, there is an irreducible
open three-manifold with the end set homeomorphic to a Cantor set
and with the end homogeneity group isomorphic to $G$.  (See
Corollary \ref{maincorollary}.)

The Cantor sets produced are \emph{unsplittable}, 
in the sense that for each such $C$, no $2$-sphere in the complement of $C$
separates points of $C$. We produce these examples  by
constructing,  for each natural number $m$ greater than one, $3$-manifolds with end homogeneity
groups $\mathbb{Z}_m$,  and by separately constructing $3$-manifolds with  end homogeneity group $\mathbb{Z}$. We then link the
Cantor sets needed for a given abelian group in an unsplittable
manner.

In Section \ref{backgroundsection}, we give definitions and the
basic results needed for the rest of the paper. In the following
section, Section \ref{Antoinesection}, we review the needed
results about Antoine Cantor sets. In Section \ref{Zpsection} we
produce Cantor sets with embedding homogeneity group
$\mathbb{Z}_{m}.$  In Section \ref{Zsection} we produce Cantor
sets with embedding homogeneity group $\mathbb{Z}$. Section
\ref{mainsection} ties together the previous results and lists
and proves the main theorems. Section \ref{questionsection} lists
some remaining questions.

\section{Preliminaries}
\label{backgroundsection}

\subsection{Background information}
Refer to \cite{GaReZe05, GaReZe06, GaReWr12} for a discussion of
Cantor sets in general and of rigid Cantor sets, and to
\cite{Ze05} for results about local genus of points in Cantor
sets and defining sequences for Cantor sets.  The bibliographies
in these papers contain additional references to results about
Cantor sets. Two Cantor sets $X$ and $Y$ in $S^3$ are  said to be
\emph{equivalent} if there is a self-homeomorphism of $S^3$ 
taking $X$ to $Y$. If there is no such homeomorphism, the Cantor
sets are said to be \emph{inequivalent}, or \emph{inequivalently
embedded}. A Cantor set $C$ is \emph{rigidly embedded} in $S^{3}$
if the only self-homeomorphism of $C$ that extends to a
homeomorphism of $S^{3}$ is the identity. 

\subsection{Geometric index}

We list the results we need on geometric index. See Schubert
\cite{Sc53} and \cite{GaReWrZe11}  for more details.

If $K$ is a link in the interior of a solid torus $T$,  the \emph
{geometric index} of $K$ in $T$ is denoted by $\text{N}(K,T)$. 
This \emph{geometric index} is the minimum  of $|K \cap D|$ over
all meridional disks  $D$ of $T$ intersecting $K$ transversely.   
 If $T$ is a solid torus and
$M$ is a finite union of disjoint solid tori so that $M \subset
\text{Int}(T)$, then the geometric index $\text{N}( M,T)$ of
$M$ in  $T$ is $\text{N}(K,T)$ where $K$ is a core of $M$.

\begin{thm}\label{indexone}  {\rm(}\cite{Sc53}, \cite[Theorem
3.1]{GaReWrZe11}{\rm)} Let $T_0$ and $T_1$ be unknotted solid
tori in $S^{3}$ with  $T_0 \subset \rm{Int}  (T_1)$ and $\rm{N}(
T_0, T_1) = 1$.  Then $\partial T_0$ and  $\partial T_1$ are
parallel; i.e., the manifold $T_1 -  \rm{Int}  (T_0)$ is
homeomorphic to $\partial T_0 \times I$ where $I$ is the closed
unit interval $[0,1]$. \end{thm}

\begin{thm}\label{productindex}  {\rm(}\cite{Sc53}, \cite[Theorem
3.2]{GaReWrZe11}{\rm)} Let $T_0$ be a finite union of disjoint
solid tori. Let $T_1$ and $T_2$ be solid tori so that $T_0
\subset \rm{Int} ( T_1)$ and $T_1 \subset \rm{Int}  (T_2)$.  Then
$\rm{N}(T_0, T_2) =  \rm{N}(T_0, T_1) \cdot  \rm{N}(T_1, T_2)$.
\end{thm}

There is one additional result we will need.

\begin{thm}\label{evenindex}{\rm(}\cite{Sc53}, \cite[Theorem
3.3]{GaReWrZe11}{\rm)} Let $T$ be a solid torus in $S^{3}$ and
let $T_{1},\ldots T_{n}$ be unknotted pairwise disjoint solid tori in 
$T$, each of
geometric index $0$ in $T$. Then the geometric index of
$\bigcup\limits_{i=1}^{n}T_{i}$ in $T$ is even. \end{thm}

\subsection{Defining sequences and local genus}

We review the definition and some basic facts from \cite{Ze05}
about  the local genus of points in a Cantor set. See \cite{Ze05}
for a discussion of defining sequences.

Let
${ \mathcal{D}}(X)$ be the set of all defining sequences for a
Cantor set $X\subset S^3$.
Let $(M_i)\in{\mathcal{D}}(X)$ be a specific defining sequence for a $X$. For  $A\subset X$, denote by $M_i^A$
the union of those components of $M_i$ which intersect $A$. The genus of $M_i^A$, 
$g(M_i^A)$, is the maximum genus of a component of $M_i^A$. Define 
\[ 
g_A(X;(M_i)) = \sup\{g(M_i^A);\ i\geq 0\}\ \ \mbox{ and}
\] 
\[ 
g_A(X) = \inf\{ g_A(X;(M_i));\ (M_i) \in
{\mathcal{D}}(X)\}. 
\] 
The number $g_A(X)$ is called \emph{the
genus of the Cantor set $X$ with respect to the subset $A$}. For
$A=\{x\}$ we call the number $g_{\{x\}}(X)$  \emph{the local
genus of the Cantor set $X$ at the point $x$} and denote it by
$g_x(X)$.

Let $x$ be an arbitrary point of a Cantor set $X$ and $h\colon
S^3\to S^3$ a homeomorphism. Then the local genus $g_x(X)$ is the
same as the local genus $g_{h(x)}(h(X))$. Also note that if $x\in
C\subset C^{\prime}$, then the local genus of $x$ in $C$ is less
than or equal to the local genus of $x$ in $C^{\prime}$.  See
\cite[Theorem 2.4]{Ze05}.

The following result from \cite{Ze05} is needed to show that
certain points in our examples have local genus $2$.

\begin{thm}\cite{Ze05} \label{Slicing} 
Let $X,Y\subset S^3$ be
Cantor sets and $p\in X\cap Y$.  Suppose there exists a 3-ball
$B$ and a 2-disc $D\subset B $ such that

\begin{enumerate}
\item $p\in\rm{Int} (B)$, $\partial D=D\cap\partial B$, $D\cap (X\cup
Y)=\{p\}$; and

\item $X\cap B\subset B_X\cup\{p\}$ and $Y\cap B\subset
B_Y\cup\{p\}$ where $B_X$ and $B_Y$ are the components of
$B -  D$.
\end{enumerate}

Then $g_p(X\cup Y)=g_p(X)+g_p(Y)$.
\end{thm}

\subsection{Discussion and examples of ends and homogeneity groups}
For background on Freudenthal compactifications and theory of
ends, see \cite{Di68}, \cite{Fr42}, and \cite{Si65}. For an
alternate proof using defining sequences of the result that every
homeomorphism of the open $3$-manifold extends to a homeomorphism
of its Freudenthal compactification, see \cite{GaRe12}.

At the end of the next section, we will discuss elements of the
homogeneity group of a standard self-similar Antoine Cantor set.
Note that removing $n$ points from $S^3$ yields a reducible
open $3$-manifold with end homogeneity group the symmetric group
on $n$ elements. It is not immediately obvious how to produce
examples that are irreducible, have a rich end structure (for
example a Cantor set), and at the same time have specified
abelian end homogeneity groups.

\section{Antoine Cantor set Properties}
\label{Antoinesection}

An Antoine Cantor set is described by a defining sequence $(M_i)$ as follows. Let $M_{0}$ be an
unknotted solid torus in $S^{3}$. Let $M_{1}$ be a chain of at
least four linked, pairwise disjoint, unknotted solid tori in $M_{0}$ 
as in Figure
\ref{FiniteChain}. Inductively, $M_{i}$ consists of pairwise
disjoint solid tori in $S^{3}$ and $M_{i+1}$ is obtained from
$M_{i}$ by placing a chain of at least four linked, pairwise disjoint, unknotted
solid tori in each component of $M_{i}$. If the diameter of the
components goes to $0$, the Antoine Cantor set is
$C=\bigcap\limits_{i=0}^{\infty}M_{i}$.

We refer to Sher's paper \cite{Sh68}  for basic results and
description of Antoine Cantor sets. The key result we shall need is the
following:
\begin{thm}\cite[Theorems 1 and 2]{Sh68}\label{Sher} 
Suppose $C$
and $D$ are Antoine Cantor sets in $S^{3}$ with defining
sequences $(M_{i})$ and $(N_{i})$ respectively. The Cantor
sets are equivalently embedded if and only if there is a 
self-homeomorphism $h$ of $S^{3}$ with $h(M_{i})=N_{i}$ for each $i$.
\end{thm}

In particular, the number and adjacency of links in the chains
must match up at each stage. Because we need a modification of
this result for infinite chains in Section \ref{Zsection}, we
provide below an outline of an alternative proof of the forward
implication.

\begin{proof} (Forward implication of Theorem \ref{Sher}.) It
suffices to show if $C$ has two Antoine defining sequences
$(M_{i})$ and $(N_{i})$, then there is a homeomorphism $h$ as
in the theorem.

\textbf{Step 1: }There is a general position homeomorphism
$h_{1}$, fixed on $C$, so that $h_{1}(\partial (M_{1})\cup
\partial(M_{2}))$ is in general position with
$\partial(N_{1})\cup \partial( N_{2})$. The curves of
intersection of $h_{1}(\partial (M_{1})\cup \partial(M_{2})) \cap
\left(\partial(N_{1})\cup \partial( N_{2})\right)$ can be
eliminated by a homeomorphism $h_{2}$ also fixed on $C$ by a
standard argument and the facts that any nontrivial curve on
$\partial(M_{i})$ does not bound a disc in the complement of $C$,
and that no $2$-sphere separates the points of $C$. For details
on the type of argument in this step, see \cite{Sh68} or
\cite{GaReWrZe11}.

\textbf{Step 2:} Let $T$ be a component of $h_{2}\circ
h_{1}(M_{1})$ and assume $T$ intersects a component $S$ of
$N_{1}$. Either $T\subset \rm{Int} (S) $ or $S\subset \rm{Int} (T)$.
First assume $T\subset \rm{Int} (S)$. If the geometric index of
$T$ in $S$ is $0$, then since the other components of $h_{2}\circ
h_{1}(M_{1})$ are linked to $T$ by a finite chain, all components
of $h_{2}\circ h_{1}(M_{1})$ are in the interior of $S$. This is
a contradiction since there are points of $C$ not in $S$. So the
geometric index of $T$ in $S$ is greater than or equal to $1$.

Note that $T$ cannot be contained in any component of $N_{2}$
that is in $S$ since these have geometric index $0$ in $S$. So
$T$ contains all the components of $N_{2}$ that are in $S$. Each
of these components has geometric index $0$ in $T$, so the union
of these components has an even geometric index in $T$ by Theorem
\ref{evenindex}. This geometric index must then be $2$ and the
geometric index of $T$ in $S$ must be $1$. Now there is a
homeomorphism $h_{3}$ fixed on $C$ and the complement of $S$ that
takes $T$ to $S$.

If instead $S\subset \rm{Int} (T)$, a similar argument shows there
is a homeomorphism $h_{3}$ fixed on $C$ and the complement of $T$
taking $S$ to $T$. The net result is that it is possible to
construct a homeomorphism $h_{3}^{\,\prime}$ taking the
components of $h_{2}\circ h_{1}(M_{1})$ to the components of $S$.
One now proceeds inductively, matching up further stages in the
constructions, obtaining the desired homeomorphism $h$ as a limit.
\end{proof}

\begin{remark}
\label{separation}
A standard argument shows that an Antoine Cantor set cannot be
separated by a $2$-sphere. This is also true if the 
construction starts with a finite open chain of linked tori as in
Figure \ref{GeneralChain}. 
\end{remark}

\begin{remark}
\label{isotopy} Also note that the homeomorphism of Theorem
\ref{Sher} can be realized as the final stage of an isotopy since
each of the homeomorphisms in the argument can be realized by an
isotopy.
\end{remark}

\subsection{Homogeneity groups of Antoine Cantor sets}
Let $C$ be obtained by a standard Antoine construction where the
same number of tori are used in tori of previous stages in each
stage of the construction. For example, the Antoine pattern in
Figure \ref{FiniteChain} with $24$ smaller tori, each geometrically similar to
the outer torus, can be repeated in each component at each stage
of the construction. 

We now consider
some elements of the embedding homogeneity group of $C$.
There is an obvious $\mathbb{Z}_{24}$ group
action on the resulting Cantor set obtained by rotating and
twisting the large torus. There is also $\mathbb{Z}_{24}\oplus
\mathbb{Z}_{24}$ action on $C$ by considering the first two stages, 
where we require each torus in the second stage to
rotate the same amount. If we allow the tori in the second stage to
rotate different amounts, we get an even larger group action by 
$\mathbb{Z}_{24}$ wreath product with $\mathbb{Z}_{24}$. 
Considering more stages results in even more
complicated group actions.

In addition to these group actions by rotating and twisting, there are also orientation 
reversing $\mathbb{Z}_{2}$ actions that
arise from reflecting through a horizontal plane (containing the core
of the large torus) or through a vertical plane (containing meridians
of the large torus).

From this we see that even for a simple self-similar Antoine Cantor
set, the embedding homogeneity group is more complex than just the
group of obvious rotations from the linking structure. In the next
section we shall carefully combine certain Antoine constructions to
produce a more rigid example with nontrivial end homogeneity group
in such a way that these kinds of orientation reversing homeomorphisms are not possible, 
and that also restricts the possible rotations.

\section{A Cantor Set with embedding homogeneity group 
$\mathbb{Z}_{m}$} \label{Zpsection}

Fix an integer $m>1$. We describe how to construct a Cantor set
in $S^{3}$ with embedding homogeneity group $\mathbb{Z}_{m}$.

\begin{const} \label{ZpConst} As in the previous section, let
$S_{0}$ be an unknotted solid torus in $S^{3}$. Let $\{S_{(1,i)}
\, \vert\, 1\leq i \leq 4m \}$, be an Antoine chain of $4m$
pairwise disjoint linked solid tori in the interior of $S_{0}$ and let 
\[
S_{1}=\bigcup_{i=1}^{4m}S_{(1,i)} .
\]
\begin{figure}[htb]
\begin{center}
\includegraphics[width=.8\textwidth]{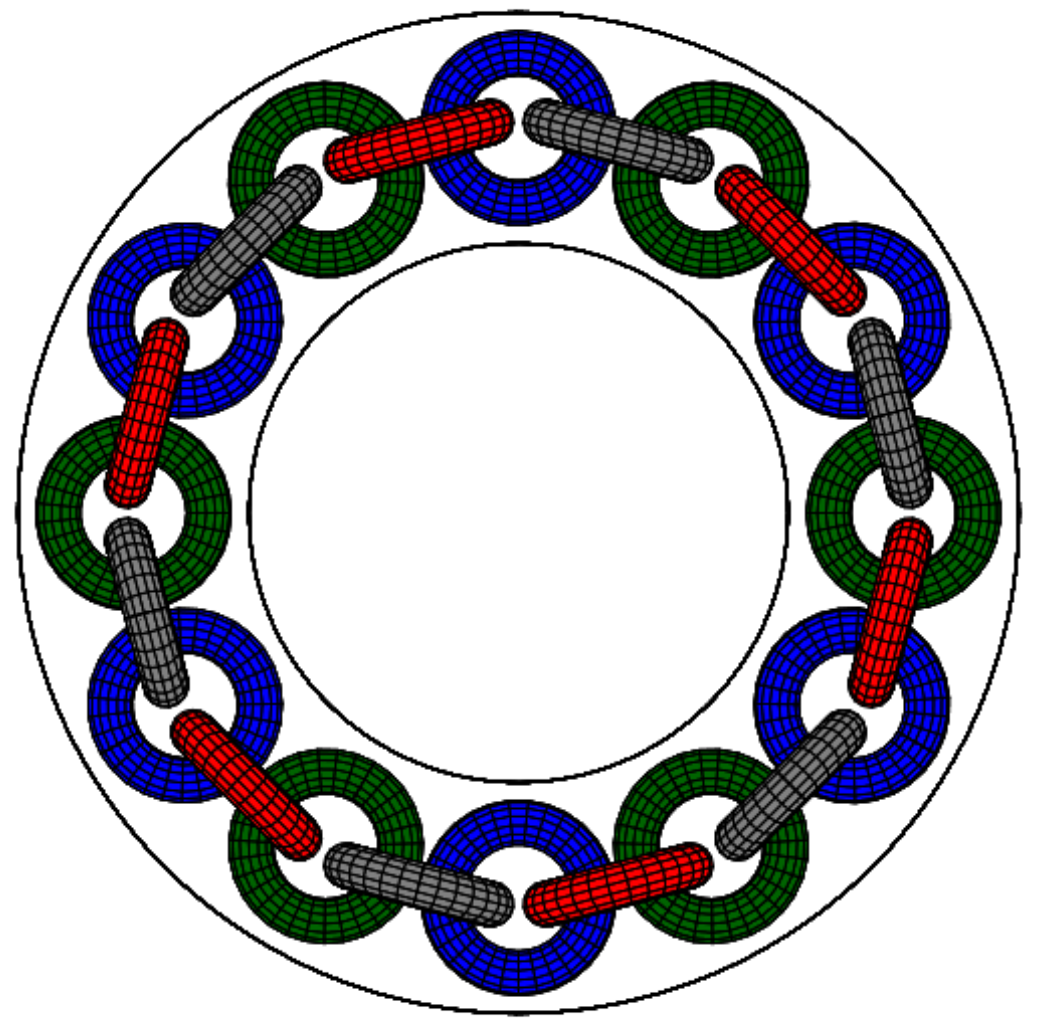}
\end{center}
\caption{ Antoine Chain With $\mathbb{Z}_{6}$ Group Action}
\label{FiniteChain}
\end{figure}

See Figure \ref{FiniteChain} for the case when $m=6$. Let
$C_{j},1\leq j \leq4,$ be a rigid Antoine Cantor set with first
stage $S_{(1,j)}$. Choose these four rigid Antoine Cantor sets so
that they are inequivalently embedded in $S^{3}$. Let $h$ be a
homeomorphism of $S^{3}$, fixed on the complement of the interior
of $S_{0}$, that takes $S_{(1,j)}$ to $S_{(1,j+4\mod 4m)}$ for
$1\leq j \leq 4m$. Require that $h^{m}$ is the identity on each
$S_{(1,i)}$.

For $4k < i \leq 4k+4$, let $C_{i}$ be the rigid Cantor set in
$S_{(1,i)}$ given by $h^{k}(C_{i-4k})$. Note that this produces
$m$ copies of each of the rigid Cantor sets $C_{1}$, $C_{2}$,
$C_{3}$, and $C_{4}$. Again, see Figure \ref{FiniteChain} where
the shading indicates the four classes of rigid Cantor sets. The
Cantor set we are looking for is
\[
C=\bigcup_{i=1}^{4m} C_{i}.
\]
\end{const}

\begin{thm} 
\label{ZpTheorem}
The Cantor set $C$ from the previous construction has embedding
homogeneity group $\mathbb{Z}_{m}$ and is unsplittable. \end{thm} \begin{proof} Let
$\ell:S^{3}\rightarrow S^{3}$ be a homeomorphism taking $C$ to
$C$. We show that $\ell\vert_{C}=h^{k}\vert_{C}$ for some $k$,
$1\leq k \leq m$. By Sher \cite{Sh68}, we may assume that $\ell$
takes each $S_{(1,i)}$ to some $S_{(1,j)}$, and so
$\ell(C_{i})=C_{j}$. Because of the distinct rigid Cantor sets
involved, this is only possible if $j-i\equiv 0 \mod 4$.

So assume that $\ell(S_{(1,1)})=S_{(1,4k+1)}$. Then
$\ell(S_{(1,2)})$ must be one of the two tori linked with
$S_{(1,4k+1)}$, hence either $S_{(1,4k)}$ or $S_{(1,4k+2)}$. Since $(4k
-2) \not \equiv  0 \mod 4$, $\ell(S_{(1,2)})$ must be $S_{(1,4k+2)}$.
Continuing inductively, one sees that
$\ell(S_{(1,i)})=S_{(1,4k+i\mod 4m)}$. Thus
$\ell(C_{i})=C_{4k+i\mod 4m}$. But $h^{k}(C_{i})$ is also 
$C_{4k+i\mod 4m}$. Since these are rigid Cantor sets,
$\ell\vert_{C_{i}}=h^{k}\vert_{C_{i}}$ for each $i$.

So the embedding homogeneity group of $C$ is $\{h^{k}\vert 1\leq
k \leq p\}\simeq \mathbb{Z}_{m}$. By Remark \ref{separation}, $C$ is unsplittable. The assertion follows.
\end{proof}

\section{A Cantor Set with embedding homogeneity group $\mathbb{Z}$}
\label{Zsection}

We now construct a Cantor set in $S^{3}$ with embedding
homogeneity group $\mathbb{Z}$. This requires careful analysis of an infinite chain
analogue of the Antoine construction. 

\begin{const}
\label{Zconst}
Let $S_{0}$ be a pinched solid torus in $S^{3}$, i.e., the quotient of a
solid torus with a meridional disk collapsed to a single point $w$. Let
$T_{i}, i\in {\mathbb{Z}}$ be a countable collection of unknotted
pairwise disjoint solid tori in $S_{0}$ so that each $T_{i}$ is of simple linking
type with both $T_{i-1}$ and $T_{i+1}$, and  is  not linked with
$T_{j}, j\neq i-1\rm{\  or\ }i+1$. Place the tori $T_{i}$ so that 
the $T_{i}, i>0$ and the $T_{i}, i<0$ have $w$ as a limit point
as in Figure \ref{InfiniteChain}.

\begin{figure}[htb]
\begin{center}
\includegraphics[width=.75\textwidth]{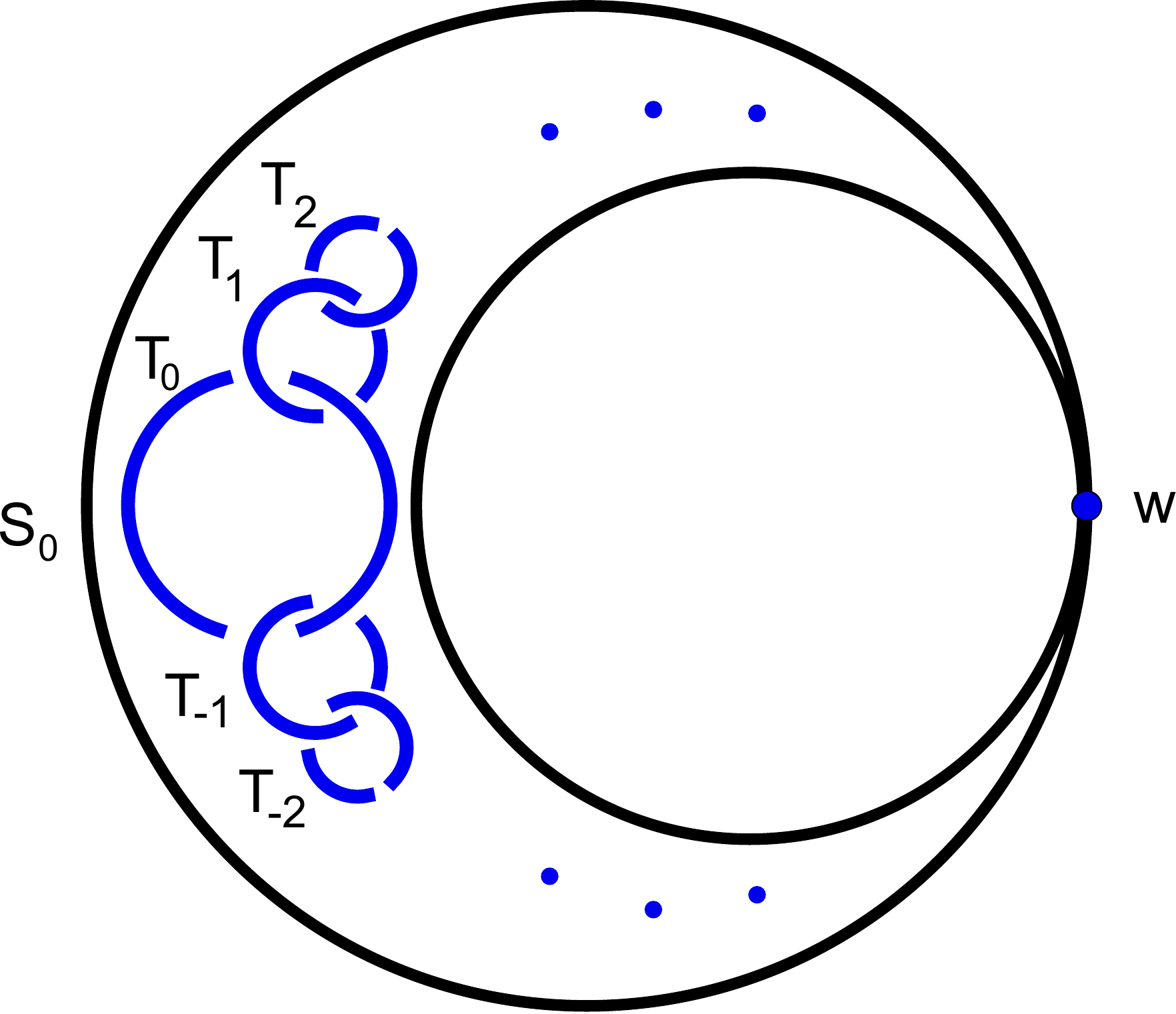}
\end{center}
\caption{Infinite Antoine Chain}
\label{InfiniteChain}
\end{figure}

Let $C_{j},1\leq j \leq3,$ be a rigid Antoine Cantor set with
first stage $T_{j}$. Choose these three rigid Antoine Cantor sets
so that they are inequivalently embedded in $S^{3}$. Let $\alpha$
be a homeomorphism of $S^{3}$, fixed on the complement of the
interior of $S_{0}$, that takes $T_{j}$ to $T_{j+3}$ for $j\in
\mathbb{Z}$.

For $3k < i \leq 3k+3$, let $C_{i}$ be the rigid Cantor set in
$T_{i}$ given by $\alpha^{k}(C_{i-3k})$. Note that this produces
a countable number of copies of each of the rigid Cantor sets
$C_{1}$, $C_{2}$, and $C_{3}$. Again, see Figure
\ref{InfiniteChain}. The Cantor set we are looking for is
\[
C=\overline{\bigcup_{i\in{\mathbb{Z}}} C_{i}}=\bigcup_{i\in{\mathbb{Z}}}
C_{i}\cup \{w\}.
\]
\end{const}
Note that $C$ is a Cantor set since it is perfect, compact, and totally disconnected.
\begin{thm} 
\label{ZTheorem}
The Cantor set $C$ from the previous construction has embedding
homogeneity group $\mathbb{Z}$ and is unsplittable.
\end{thm}

\begin{proof}
It is clear from the construction that each point of 
$C -  \{w\}$ has local genus $1$. 
 Theorem \ref{Slicing} applied to $w$ and the Cantor sets 
 $C_{+}=\overline{\bigcup_{i>0} C_{i}}$ and 
 $C_{-}=\overline{\bigcup_{i<0} C_{i}}$ shows that 
 $w$ has local genus $2$ in $C$. Thus any homeomorphism of 
 $S^{3}$ 
that takes $C$ to $C$ must fix $w$.

Let $h$ be such a homeomorphism of $S^{3}$ taking $C$ to $C$. Let
$T^{\,\prime}_{i}$ be the union of the linked tori in the Antoine
chain at the second stage of the construction of $C_{i}$.
Let 
\[
\Lambda_{N}=\bigcup\limits_{i=-N}^{N}T_{i},\ 
\Gamma_{N}=\bigcup\limits_{i=-N}^{N}C_{i},  \ \rm{ and\ }
 \Lambda^{\prime}_{N}=
 \bigcup\limits_{i=-N}^{N}T_{i}^{\,\prime}.
 \]
Fix an integer $n\in \mathbb{Z}$. Since $h(T_{n})$ does not contain
$w$, there is a positive  integer $N_{1}>\vert n\vert$ such that
$h(C_{n})\subset \Gamma_{N_{1}}$.
 Similarly, there is a positive integer $N_{2}>N_{1}$ such that
$h^{-1}(\Gamma_{N_{1}})\subset \Gamma_{N_{2}}$.

As in step one in the proof of Theorem \ref{Sher},
there is a homeomorphism $k$ of $S^{3}$ to itself, fixed on $C$,
 so that 
 \[
 k\left(h\left(\partial (\Lambda_{N_2+1})\cup\partial(
\Lambda^{\prime}_{N_2+1})\right)\right)\cap
 \left(
 \partial (\Lambda_{N_2+1})\cup\partial( \Lambda^{\prime}_{N_2+1})
 \right)
 =\emptyset.
 \]
Fix a point $p$ of $C_{n}$ and let $k(h(p))=h(p)=q\in C_{m}$.
We will show that $k(h(C_{n}))=h(C_{n})=C_{m}$. Let $\ell=k\circ h$. Since $\ell(T_{n})\cap T_{m}
\neq\emptyset$, and since the boundaries do not intersect, 
either $\ell(T_{n})\subset \rm{Int}(T_{m})$ or
$\rm{Int}(\ell(T_{n}))\supset T_{m}$. We consider these cases separately. 

\textbf{Case I:} 
$\ell(T_{n})\subset \rm{Int}(T_{m})$. If $\ell(T_{n})$ has
geometric index $0$ in $T_{m}$, then $\ell(T_{n})$ is contained
in a cell in $T_{m}$ and so it contracts in $T_{m}$.  Since a contraction of
$\ell(T_{n})$ meets the boundary of the linked $\ell(T_{n+1})$, and since the boundary of $\ell(T_{n+1})$ is disjoint from the boundary of $T_{m}$, $\ell(T_{n+1})
\subset\rm{Int}(T_{m})$. Continuing inductively, one finds that one of the following two situations occur when $\ell(T_{n})$ has
geometric index $0$ in $T_{m}$:

\textbf{Case Ia:} All of
$\ell(T_{n}),\ell(T_{n+1}),\ldots \ell(T_{N_{2}})$ are contained in $T_{m}$ and have geometric index $0$ there.

\textbf{Case Ib:} There is a $j, n<j\leq N_2$, with $\ell(T_{j}) 
\subset \rm{Int}(T_{m})$ and of geometric index greater than or equal to $1$ there.

In Case Ia, it follows that $\ell(T_{N_{2}+1})\subset \rm{Int}(T_{m})$. But then $C_{m}\subset \Gamma_{N_{1}}$ and $h^{-1}(C_{m})\cap C_{N_{2}+1}\neq \emptyset$, contradicting the choice of $N_2$.

In Case Ib, suppose   $\ell(T_{j})$ has geometric index $k$ greater than $1$ in $T_{m}$.
Then, by Theorem \ref{productindex}, $\ell(T_{j})$ cannot be contained in any component  of the
next stage of the construction contained in $T_{m}$ since these
have geometric index $0$ in $T_{m}$. So some component of the
next stage in $T_{m}$ is contained in $\ell(T_{j})$ and has
geometric index $0$ there by Theorem \ref{productindex}. Since the components of the next stage are linked, all components of the next stage in
$T_{m}$ are contained in $\ell(T_{j})$. The geometric index of
the union of the next stages of in $T_{m}$ in $\ell(T_{n})$ is
even by Theorem \ref{evenindex} and can not be equal to $0$. Otherwise, by Theorem \ref{productindex} the union of the next stages of $T_{m}$
would have index $0$ in $T_{m}$, which is a contradiction. So the geometric index of 
union of the next stages of in $T_{m}$ in $\ell(T_{n})$ is at least $2$. Then by Theorem
\ref{productindex}, the geometric index of union of the next stages of in $T_{m}$ in $T_{m}$ is at least $4$, contradicting the fact that this geometric index is $2$.

It follows that $\ell(T_{j})$ has geometric index $1$ in $T_{m}$
and contains the union of the next stages contained in $T_{m}$.
Since $\ell$ is a homeomorphism that takes $C$ to $C$, 
it follows from the construction of $C$ that $\ell(C_{j})=C_{m}$. Since $\ell(p)\in C_{m}$,
$\ell(T_{n})\cap \ell(C_j)\neq \emptyset$, contradicting the fact that $\ell$ is a homeomorphism.  

Thus, neither Case Ia nor Case Ib can occur. So the geometric index of $\ell(T_{n})$ in $T_{m}$ must be at least $1$. Repeating the argument from Case Ib above, with $T_{j}$ replaced by $T_{n}$, we see that $\ell(T_{n})$ has geometric index $1$ in $T_{m}$
and contains the union of the next stages contained in $T_{m}$.
Since $\ell$ is a homeomorphism that takes $C$ to $C$, 
it follows from the construction of $C$ that $\ell(C_{n})=C_{m}$ as desired.

\textbf{Case II:}  $\rm{Int}(\ell(T_{n}))\supset T_{m}$. 
Then ${\ell^{-1}}(T_{m})\subset \rm{ Int}(T_{n})$. The argument
from Case I can now be repeated replacing $\ell$ by $\ell^{-1}$
and interchanging $T_{n}$ and $T_{m}$. It follows that
$\ell^{-1}(C_{m})=C_{n}$ and so $\ell(C_{n})=C_{m}$ as desired.

Since $\ell(C_{n})=h(C_{n})=C_{m}$, it must be the case that $(m-n)\equiv 0\mod 3$.
Continuing as in the proof of the $\mathbb{Z}_{m}$ result (Theorem \ref{ZpTheorem}), we have that
for each $i$, $h(C_{i})=C_{i+(m-n)}$. Recall that for the homeomorphism $\alpha$ from the construction of $C$, it is  also the case that
$\alpha^{\frac{m-n}{3}}(C_{i})=C_{i+(m-n)}$. By the rigidity of
these Cantor sets, it follows that
$\alpha^{\frac{m-n}{3}}\vert_{C_{i}}=h\vert_{C_{i}}$. Thus the
embedding homogeneity group of $C$ is $\{ \alpha^{k} \vert k\in
\mathbb{Z}\} \simeq \mathbb{Z}$.

We now show that $C$ is unsplittable. Assume that $\Sigma$ is a $2$-sphere in $S^{3}$ that separates $C$. Choose
$\epsilon>0$ so that the distance from $\Sigma$ to $C$ is greater
than $\epsilon$. Choose $N$ so that each $T_{i}, \vert i
\vert  \geq N,$ has diameter less than $\epsilon\slash 6$ and is
within $\epsilon\slash 6$ of $w$. Since $\Sigma$ separates $C$,
$w\cup \left(\bigcup\limits_{\vert i\vert \geq N}T_{i}\right)$ must 
be in one component of $S^{3} -  \Sigma$ and there must be
points of $C$ in the other component of  $S^{3} -  \Sigma$.
So $\bigcup\limits_{\vert i \vert \leq N} T_{i}$ contains points in both components of 
 $S^{3} -  \Sigma$.
 
Form an Antoine Cantor set $C^{\prime}$ related to $C$ as
follows. Use $\bigcup\limits_{\vert i \vert \leq N} T_{i}$ as
 a part of the first stage of the construction. Complete the first
stage of the construction by adding an unknotted solid torus $T$,
linked to $T_{N}$ and $T_{-N}$ ,  that is within the $\epsilon\slash 3$-%
neighborhood of $w$. For successive stages of the Antoine Cantor 
set $C^{\prime}$ in $T_{i}, \vert i \vert \leq N$, use the
successive stages in forming the Cantor set $C_{i}\subset C$. For
successive stages of the Antoine Cantor set $C^{\prime}$ in $T$,
use any Antoine construction.

By construction and the properties of $\Sigma$, the $2$-sphere
$\Sigma$ separates the Antoine Cantor set $C^{\prime}$,
contradicting Remark \ref{separation}.
\end{proof}

\section{Main Results}
\label{mainsection}

Given a finitely generated abelian group $G$, we use the results
from the previous two sections to construct a unsplittable
Cantor set $C_{G}$ in $S^{3}$ with embedding homogeneity group
$G$.

\begin{const}
\label{MainConstruction}
Let $G\simeq \mathbb{Z}^{n}\oplus \mathbb{Z}_{m_{1}}\oplus
\mathbb{Z}_{m_{2}}\cdots \oplus \mathbb{Z}_{m_{k}}$ be any finitely
generated abelian group. Form a simple chain of $n+k$ pairwise disjoint unknotted
solid tori. Figure \ref{GeneralChain} illustrates the case
$n+k=6$. Label the tori as $T_{1},T_{2},\ldots T_{n+k}$ so that
$T_{1}$ is only linked with $T_{2}$, $T_{n+k}$ is only linked
with $T_{n+k-1}$, and so that $T_{i}$, $2\leq i \leq n+k-1$, is
linked with $T_{i-1}$ and $T_{i+1}$. 

\begin{figure}[htb]
\begin{center}
\includegraphics[width=.65\textwidth]{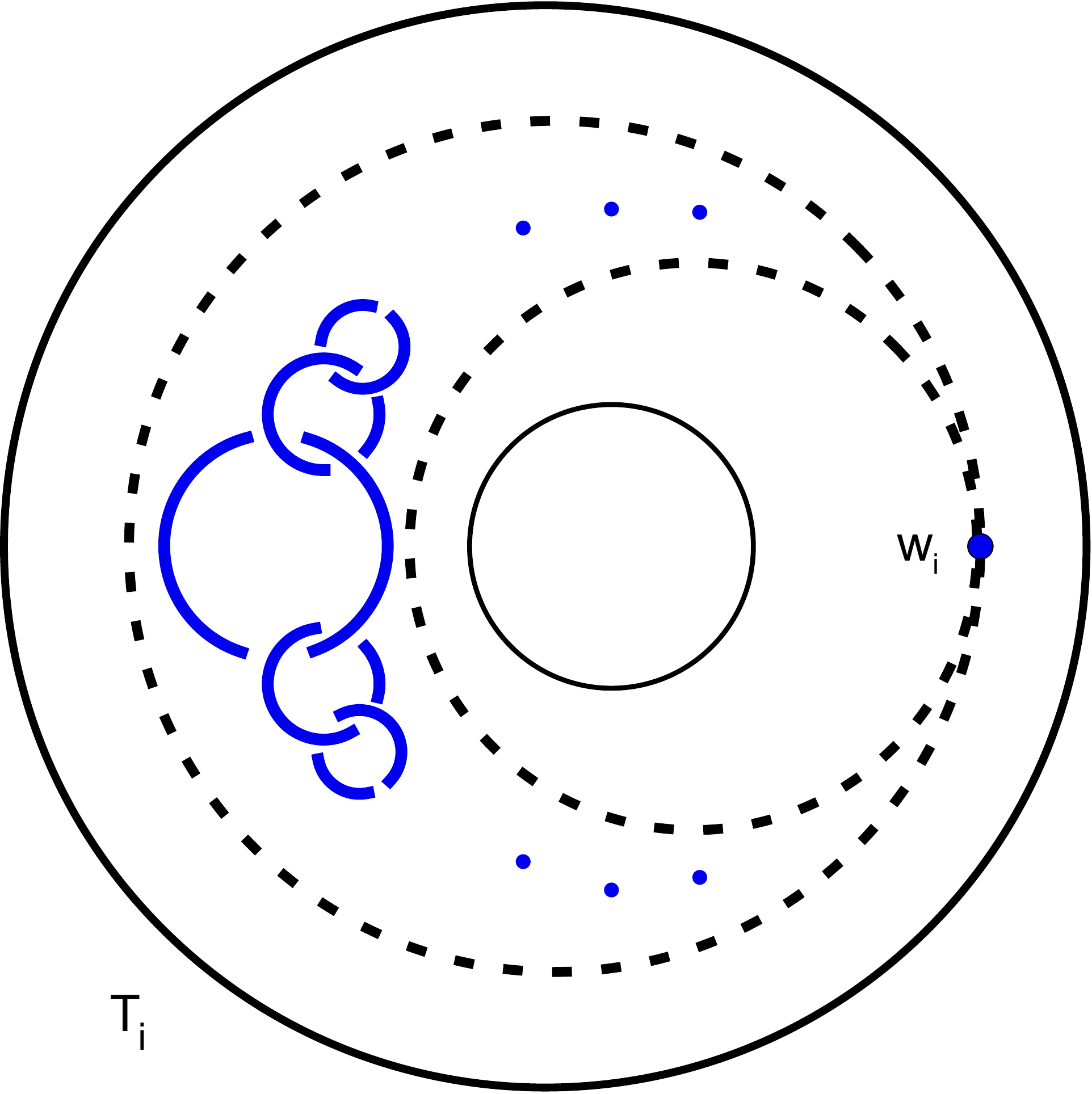}
\end{center}
\caption{Pinched Torus in  $T_{i}$}
\label{PinchedTorus}
\end{figure}

For  $1\leq i\leq n$, perform Construction \ref{Zconst} in
$T_{i}$, treating a pinched version of $T_{i}$ in the interior of
$T_{i}$ as the torus $S_{0}$ in Construction \ref{Zconst}. Let
$w_{i}$ be the limit point corresponding to $w$ in Construction
\ref{Zconst}. This yields a Cantor set $C_{i}$ in $T_{i}$ with
embedding homogeneity group $\mathbb{Z}$. See Figure \ref{PinchedTorus}.

For $n+1\leq i \leq n+k$, perform Construction \ref{ZpConst} for
the group $\mathbb{Z}_{m_{i-n}}$ in $T_{i}$. This yields a Cantor
set $C_{i}$ in $T_{i}$ with embedding homogeneity group
$\mathbb{Z}_{m_{i-n}}$. Choose all the rigid Cantor sets from
Constructions  \ref{Zconst} and \ref{ZpConst} to be inequivalent.

Let 
\[
C_{G}=\bigcup_{i=1}^{n+k}C_{i} .
\]
\begin{figure}[htb]
\begin{center}
\includegraphics[width=.75\textwidth]{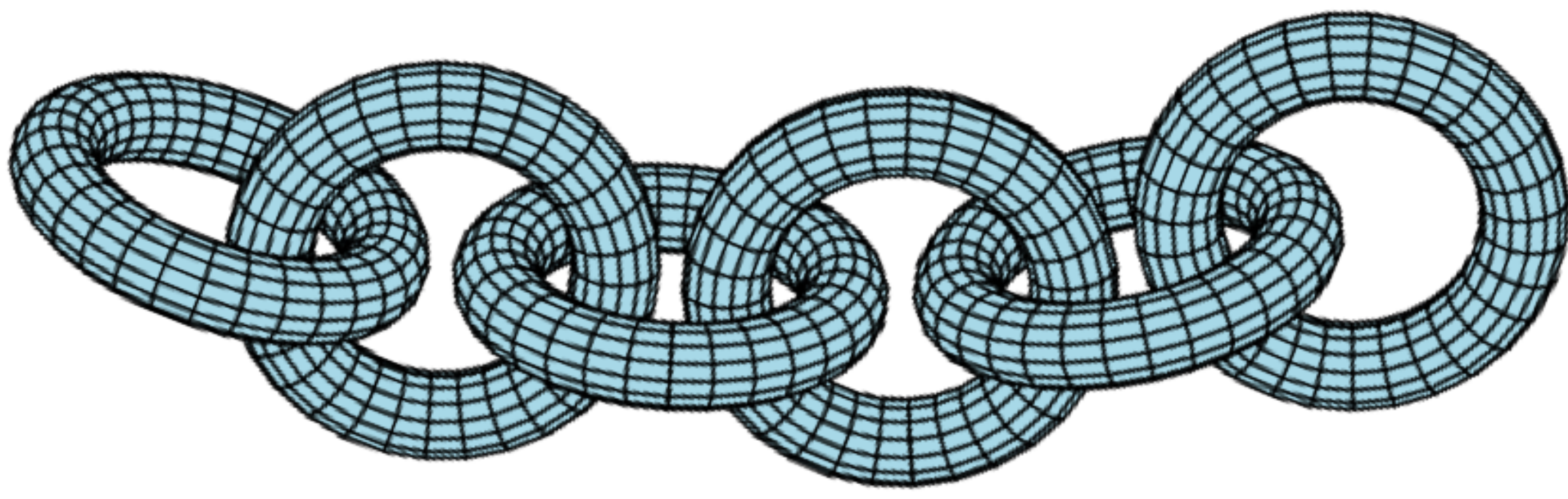}
\end{center}
\caption{Example of Antoine Chain Containing $C_{G}$}
\label{GeneralChain}
\end{figure}

\end{const}

\begin{thm} 
The Cantor set $C_{G}$ constructed above has embedding
homogeneity group $G$ and is unsplittable.
\label{maintheorem}
\end{thm}
\begin{proof}
For each $i, 1\leq i \leq n+k$, let $h_{i}$ be a self-%
homeomorphism of $S^{3}$, fixed on the complement of $T_{i}$,
such that $h_{i}\vert_{C_{i}}$ generates the embedding
homeomorphism group of $C_{i}$ ($\mathbb{Z}$ for $1\leq i\leq n$
and $\mathbb{Z}_{m_{i-n}}$ for $n+1\leq i \leq n+k$). Then:
\[
\left\{\left.\left(h_{1}^{j_{1}}\circ h_{2}^{j_{2}}\ldots \circ
h_{n+k}^{j_{n+k}}
\right)\right\vert_{C_{G}}\right\}
\simeq 
G\simeq \mathbb{Z}^{n}\oplus \mathbb{Z}_{m_{1}}\oplus
\mathbb{Z}_{m_{2}}\cdots \oplus \mathbb{Z}_{m_{k}}.
\]
Let $h$ be a homeomorphism of $S^{3}$ to itself taking $C_{G}$ to
$C_{G}$. We will show that $h\vert_{C_{G}}=
\left.\left(h_{1}^{j_{1}}\circ h_{2}^{j_{2}}\ldots \circ
h_{n+k}^{j_{n+k}} \right)\right\vert_{C_{G}}$ for some choice of
$j_{i}$. \\
\textbf{Step 1: }The homeomorphism $h$ must take  each
$C_{i}$ to itself. As in the proof of Theorem \ref{ZTheorem},
there are exactly $n$ points of genus $2$ in $C_{G}$, one in each
$C_{i}$, $1\leq  i \leq n$. These are the points
$\{w_{1},w_{2},\ldots w_{n}\}$. The homeomorphism must take this
set of genus $2$ points to itself.

Let $T$ be one of the solid torus components of the first stage
of the Antoine construction for some $C_{i}$, $1\leq i\leq n+k$.
As in the proofs of Theorems \ref{ZpTheorem} and \ref{ZTheorem},
after a general position adjustment, $h(T)$ must either lie in
the interior of some solid torus component $T^{\,\prime}$ of the
first stage of the Antoine construction for some $C_{j}$, or
$T^{\,\prime}$ must lie in the interior of $h(T)$. A similar
argument to that for Theorem \ref{ZTheorem} shows that
$N(h(T),T^{\prime})=1$ or $N(T^{\prime},h(T))=1$ and that
$h(C_{i}\cap T)=C_{j}\cap T^{\,\prime}$.

This same argument can be applied to all  first stage tori in
$C_{i}$, resulting in the fact that $h(C_{i})=C_{j}$. Because of
the inequivalence of the rigid Cantor sets used in the
construction, $i=j$ and $h(C_{i})=C_{i}$.

\textbf{Step 2:} For each $i$, $h\vert
_{C_{i}}=h_{i}^{k(i)}\vert_{C_{i}}$ for some $k(i)$. By Step 1, we have
that $h(C_{i})=C_{i}$. It follows from the construction that
$h\vert _{C_{i}}=h_{i}^{k(i)}$ for some $k(i)$. From this, it
follows that $h\vert_{C_{G}}= \left.\left(h_{1}^{j_{1}}\circ
h_{2}^{j_{2}}\ldots \circ h_{n+k}^{j_{n+k}}
\right)\right\vert_{C_{G}}$ for some choice of $j_{i}$.

Thus, the embedding homeomorphism group of $C_{G}$ is isomorphic
to $G$.

\textbf{Step 3:} $C_{G}$ is unsplittable. Let $\Sigma$ be a
$2$-sphere in $S^{3}$ separating $C_{G}$. As in the proof of
Theorem \ref{ZTheorem}, an Antoine Cantor set with first stage
$\bigcup\limits_{i=1}^{n+k} T_{i}$ can be formed so that $\Sigma$
separates this Antoine Cantor set. This is a contradiction. See
Remark \ref{separation}.
\end{proof}

\begin{cor}\label{maincorollary}
Let $G\simeq \mathbb{Z}^{n}\oplus \mathbb{Z}_{m_{1}}\oplus
\mathbb{Z}_{m_{2}}\cdots \oplus \mathbb{Z}_{m_{k}}$ be any finitely
generated abelian group. Then there is a irreducible open
$3$-manifold $M_{G}$ whose Freudenthal compactification  is $S^{3}$
with the following properties: 
\begin{itemize} 
\item the end set is 
homeomorphic to a Cantor set, 
\item the end homogeneity group of
$M_{G}$ is isomorphic to $G$, and

\item $M_{G}$ is genus one at infinity except for $n$ ends where
it is genus two at infinity.
\end{itemize}
\end{cor}
\begin{proof}
Let $M_{G}$ be $S^{3} -  C_{G}$ where $C_{G}$ is as in
Construction \ref{MainConstruction}. Then the end set of $M_{G}$
is $C_{G}$ and the end homogeneity group of $M_{G}$ is isomorphic
to the embedding homogeneity group of $C_{G}$. $M_{G}$ is
irreducible because $C_{G}$ is unsplittable. The first two
claims now follow from Theorem \ref{maintheorem}.  The third
claim follows from the proof of this theorem.
\end{proof}

\begin{remark}
Note that for each finitely generated abelian group $G$ as above,
there are uncountably many non-homeomorphic $3$-manifolds as in
the corollary. This follows from varying the rigid Cantor sets
used in the construction.
\end{remark}

\section{Questions}
\label{questionsection}
The following questions arise from a consideration of the results
in this paper.

\begin{que} If a finitely generated abelian group is infinite, is
there an open $3$-manifold with end homogeneity group $G$ that is
genus one at infinity?
\end{que}
\begin{que}Given a finitely generated abelian group $G$, are
there simply connected open $3$-manifolds with end homogeneity
group $G$?
\end{que}
\begin{que}
Is the mapping class group of the open $3$-manifold $M_{G}$
isomorphic to $G$?
\end{que}
\begin{que}
If $G$ is a finitely generated non-Abelian group, is there an open $3$-manifold with end homogeneity
group $G$?
\end{que}
\begin{que}
If $G$ is a non-finitely generated group, is there an open $3$-manifold with end homogeneity
group $G$?
\end{que}


\section{Acknowledgments}
The authors would like to thank the referee for a number of helpful suggestions, including a clarification of the proof of Theorem \ref{ZTheorem}. The first author was supported in part by the National Science
Foundation grants DMS0852030 and DMS1005906. Both authors were
supported in part by the Slovenian Research Agency grant
BI-US/11-12-023. The second author was supported in part by the
Slovenian Research Agency grants  P1-0292-0101 and J1-2057-0101.

\bibliographystyle{amsalpha}
\def\cprime{$'$}
\providecommand{\bysame}{\leavevmode\hbox to3em{\hrulefill}\thinspace}
\providecommand{\MR}{\relax\ifhmode\unskip\space\fi MR }
\providecommand{\MRhref}[2]{%
  \href{http://www.ams.org/mathscinet-getitem?mr=#1}{#2}
}
\providecommand{\href}[2]{#2}

\end{document}